\documentclass[a4paper,twoside,10pt,reqno]{article}
\usepackage{graphicx}
\usepackage{epstopdf}
\usepackage[T1]{fontenc}
\usepackage[latin1]{inputenc}
\usepackage{amssymb,amsmath,amsthm,dsfont,mathrsfs}
\usepackage{amsfonts,euscript}
\usepackage{color}
\usepackage{authblk}
\usepackage{multirow}
\usepackage{mathtools}
\usepackage{array}

\usepackage{authblk}
\usepackage{fullpage}
\usepackage{setspace}
\def\indi{{1\kern-.20em\rm I}}

\newtheorem{teo}{Theorem}[section]
\newtheorem{pro}[teo]{Proposition}

\newtheorem{ex}{Example}[section]

\newcommand{\PreserveBackslash}[1]{\let\temp=\\#1\let\\=\temp}
\newcolumntype{C}[1]{>{\PreserveBackslash\centering}p{#1}}
\newcolumntype{R}[1]{>{\PreserveBackslash\raggedleft}p{#1}}
\newcolumntype{L}[1]{>{\PreserveBackslash\raggedright}p{#1}}
\allowdisplaybreaks
\begin{document}

\title{pMAX Random Fields}

\author{Marta Ferreira} \affil{Centro de Matemática da Universidade do Minho; Centro de Matemática Computacional e Estocástica da Universidade de Lisboa; Centro de Estatística e Aplicações da Universidade de Lisboa, Portugal, \texttt{msferreira@math.uminho.pt}}

\author{Ana Paula Martins} \affil{Universidade da Beira Interior, Centro de Matem\'{a}tica e Aplica\c{c}\~oes (CMA-UBI), Avenida Marqu\^es d'Avila e Bolama, 6200-001 Covilh\~a, Portugal, \texttt{amartins@ubi.pt}}

\author{Helena Ferreira} \affil{Universidade da Beira Interior, Centro de Matem\'{a}tica e Aplica\c{c}\~oes (CMA-UBI), Avenida Marqu\^es d'Avila e Bolama, 6200-001 Covilh\~a, Portugal, \texttt{helenaf@ubi.pt}}
\date{}

\maketitle

\abstract{The risk of occurrence of atypical phenomena is a cross-cutting concern in several areas, such as engineering, climatology, finance, actuarial, among others. Extreme  value theory is the natural tool to approach this theme. Many of these random phenomena carry variables defined in time and space, usually modeled through random fields. Thus, the study of random fields in the context of extreme values becomes imperative and has been developed especially in the last decade. In this work, we propose a new random field, called pMAX, designed for modeling extremes.  We analyze its dependence and pre-asymptotic dependence structure through the corresponding bivariate tail dependence coefficients. Estimators for the model parameters are obtained and their finite sample properties analyzed. Examples with simulations illustrate the results.}

\graphicspath{{./Figuras/}}



\section{Introduction}

Modeling non-deterministic phenomena in a space-time context, such as precipitation values in a territory over a given period of time, can be done through random fields
$$
\{Y(x,t):\, x\in\mathbb{R}^2,\, t\in \mathbb{R}_+\},
$$
where $x$ can be interpreted as the location and $t$ the instant at which the the random amount $Y$ was recorded (see {\it{e.g.}}  Buishand {\it{et al.}}  \cite{Buish2008} 2008, Hristopulos \cite{Hris2020} 2020 and references therein).
	
Considering time discretization $t\in \mathbb{N}$,  we can write the previous stochastic process as
$$
\{Y(x,n):\,x\in\mathbb{R}^2,\,n\in\mathbb{N}\}\equiv\{Y_n(x):\,x\in\mathbb{R}^2,\,n\in\mathbb{N}\}=\bigcup_{n\geq 1}\{Y_n(x):\,x\in\mathbb{R}^2\}\,.
$$
Thus, modeling through a sequence of random fields
$$
\{Y_n(x):\,x\in\mathbb{R}^2\}_{n\geq 1}
$$
and studying their dependence and their finite marginal distributions is an  approach for the space-time analysis of random phenomena.

In this work, we intend to analyze the simultaneous occurrence of extreme values for the random variables (r.v.'s) $Y_n(x)$ and $Y_m(x')$, corresponding to any two locations $x$ and $x'$ and instants $n$ and $m$, for a sequence of random fields $\{Y_n(x):\,x\in\mathbb{R}^2\}_{n\geq 1}$, designated by pMAX fields. We say that a variable $Y_n(x)$ is pMAX, where $p$ is a reference to {\it power}, when it is of the form
\begin{eqnarray}\label{Ferreira_pmax}
Y_n(x)=X_n(x)\vee Z_n(x)^{1/\alpha(x)},\,x\in\mathbb{R}^2,
\end{eqnarray}
where $ X_n (x) $ and $ Z_n (x) $ are r.v.'s and $\alpha(x)>0$. In the literature there are models involving pMAX variables, as can be seen, for example, in Ferreira and Ferreira (\cite{Ferreira2014} 2014),  Heffernan \emph{et al.} (\cite{Heffernan2007} 2017) and references therein.

In the following section, we present the assumptions on the sequences of random fields  $\{X_n(x):\, x\in\mathbb{R}^2\}_{n\geq 1}$, $\{Z_n(x):\, x\in\mathbb{R}^2\}_{n\geq 1}$ and on the parameter function $\alpha:\mathbb{R}^2\to\mathbb{R}_+$, considered along the paper.
In Sections \ref{Ferreira_sdeptemp} and \ref{Ferreira_sdepespac} we study the temporal, spatial and spatio-temporal dependence, as well as the tendency for oscillations in the high values. Events related to extremes of $Y_n(x)$, $x\in\mathbb{R}^2$, $n\geq 1$, will concern  ``exceedances of a real level $y$'' defined by $\{Y_n(x)>y\}$. The joint occurrence of high values for $Y_n(x)$ and $Y_m(x')$ will be studied using bivariate tail dependence coefficients (Joe \cite{Joe1997} 1997). The tendency for high values oscillations will be evaluated through functions of these coefficients. Section \ref{Ferreira_sdepespac} addresses asymptotic independence in the sense of Ledford and Tawn (\cite{Tawn1996} 1996). This is a kind of weak or residual tail dependence vanishing at increasingly extreme quantiles and important for inferential purposes in order to avoid biased results. In section \ref{Ferreira_sexemp} we exemplify the results, strengthening the hypotheses for concrete models that we will simulate. Section 6 is devoted to the estimation of the model parameters $\alpha(x)>0,$ $x\in \mathbb{R}^2.$

\section{Presenting the model}\label{Ferreira_smodelo}

The pMAX random field that we propose is defined by (\ref{Ferreira_pmax}), and satisfies conditions (i), (ii) and (iii) described below:
\begin{itemize}
	\item[(i)] $\{X_n(x):\, x\in\mathbb{R}^2\}_{n\geq 1}=\{\mathbf{X}_n\}_{n\geq 1}$ is a stationary sequence of random fields with standard Fr\'echet marginals (i.e., $P(X_n(x)\leq z)=\exp(-1/z)$, $z>0$);
	\item[(ii)] $\{Z_n(x):\, x\in\mathbb{R}^2\}_{n\geq 1}=\{\mathbf{Z}_n\}_{n\geq 1}$ is a sequence of independent and identically distributed (i.i.d.)  random fields, with standard Fr\'echet marginals and independent of the previous sequence;
	\item[(iii)]  $\alpha(x)$, $x\in\mathbb{R}^2$, is a positive real valued function.
\end{itemize}\vspace{0.3cm}

For any locations $x,x'\in\mathbb{R}^2$, $ z,z'>0$ and $n,m\geq 1$, we have:
\begin{itemize}
	\item[]$P(Y_n(x)\leq z)=e^{-z^{-1}}e^{-z^{-\alpha(x)}}$,
	\item[]$P(Y_n(x)\leq z,Y_{n+m}(x')\leq z')=P(X_n(x)\leq z,X_{n+m}(x')\leq z')e^{-z^{-\alpha(x)}}e^{-z^{-\alpha(x')}},$
	\item[]$P(Y_n(x)\leq z,Y_{n}(x')\leq z')=P(X_n(x)\leq z,X_{n}(x')\leq z')P(Z_n(x)\leq z,Z_{n}(x')\leq z').$ 
\end{itemize}\vspace{0.3cm}

A sequence of pMAX random fields, $\{Y_n(x):\, x\in\mathbb{R}^2\}_{n\geq 1}=\{\mathbf{Y}_n\}_{n\geq 1}$, presents temporal and spatial dependence regulated by the function $\alpha$ and by the spatial and temporal dependence structure of $\{\mathbf{X}_n\}_{n\geq 1}$ and $\{\mathbf{Z}_n\}_{n\geq 1}$, as we shall see.

\section{Tail dependence}\label{Ferreira_sdeptemp}

We describe the spatial and temporal dependence of the pMAX model through the tail dependence coefficient  (Sibuya \cite{Sibuya1960} 1960, Joe \cite{Joe1997} 1997):
\begin{eqnarray}\label{Ferreira_lambda}
\lambda(Y_{n+r}(x')|Y_{n}(x))=\displaystyle\lim_{y\to\infty}P(Y_{n+r}(x')>y\ |\ Y_n(x)>y),
\end{eqnarray}
with $r\geq 0,$ $n\geq 1,$ $x,x'\in \mathbb{R}^2.$ \\

It will be said that $(Y_{n}(x),Y_{n+r}(x'))$ is upper tail dependent when $\lambda(Y_{n+r}(x')|Y_{n}(x))>0$ and upper tail independent if $\lambda(Y_{n+r}(x')|Y_{n}(x))=0$.

When $ r> 0 $ and $x=x',$  $\lambda(Y_{n+r}(x)|Y_{n}(x))$ summarizes the tail temporal dependence in location $x$. If  $ r= 0 $ and $x\not =x'$, then $\lambda(Y_{n}(x')|Y_{n}(x))$ measures the tail spatial dependence between locations $x$ and $x'$. When $ r> 0 $ and $x\not =x',$  $\lambda(Y_{n+r}(x')|Y_{n}(x))$ evaluates tail dependence in time and in space.

\begin{pro}\label{Ferreira_plambda}
In a pMAX random field $\{\mathbf{Y}_n\}_{n\geq 1}$, for $r>0$, we have
\begin{eqnarray}\nonumber
\lambda(Y_{n+r}(x')|Y_{n}(x))
=\left\{
\begin{array}{ll}
0 & ,\,\alpha(x)<1\\
\frac{1}{2}\lambda(X_{n+r}(x')|X_{n}(x))& ,\,\alpha(x)=1\\
\lambda(X_{n+r}(x')|X_{n}(x))& ,\,\alpha(x)>1\\
\end{array},
\right.
\end{eqnarray}
and, if $r=0$ and $x\not=x'$, then
\begin{eqnarray}\nonumber
\lambda(Y_{n}(x')|Y_{n}(x))
=\left\{
\begin{array}{ll}
\lambda(Z_n(x')^{1/\alpha(x')}|Z_n(x)^{1/\alpha(x)}) & ,\,\alpha(x)<1\\
\frac{1}{2}\left(\lambda(X_{n}(x')|X_{n}(x))+\lambda(Z_n(x')^{1/\alpha(x')}|Z_n(x)^{1/\alpha(x)})\right)& ,\,\alpha(x)=1\\
\lambda(X_{n}(x')|X_{n}(x))& ,\,\alpha(x)>1\\
\end{array}.
\right.
\end{eqnarray}
\end{pro}
\begin{proof}
Observe that
\begin{align*}\nonumber
\lefteqn{P(Y_{n+r}(x')>y\ |\ Y_n(x)>y)=\frac{1-P(Y_{n+r}(x')\leq y)-P(Y_{n}(x)\leq y)+P(Y_n(x)\leq y,Y_{n+r}(x')\leq y)}{1-P(Y_n(x)\leq y)}}\\[0.3cm]
&=  1-\frac{e^{-y^{-1}-y^{-\alpha(x')}}-P(X_n(x)\leq y,X_{n+r}(x')\leq y)P(Z_n(x)\leq y^{\alpha(x)},Z_{n+r}(x')\leq y^{\alpha(x')})}{1-e^{-y^{-1}-y^{-\alpha(x)}}}\\[0.3cm]
&= 1-\frac{e^{-y^{-1}-y^{-\alpha(x')}}-\left[P(X_{n+r}(x')\leq y)-P(X_n(x)> y,X_{n+r}(x')\leq y)\right]P(Z_n(x)\leq y^{\alpha(x)},Z_{n+r}(x')\leq y^{\alpha(x')})}{1-e^{-y^{-1}-y^{-\alpha(x)}}}\\[0.3cm]
&=  1-e^{-y^{-1}-y^{-\alpha(x')}}\frac{1-\frac{P(Z_n(x)\leq y^{\alpha(x)},Z_{n+r}(x')\leq y^{\alpha(x')})}{e^{-y^{-\alpha(x')}}}}{1-e^{-y^{-1}-y^{-\alpha(x)}}}\\
&\ \ \ -P(X_{n+r}(x')\leq y\ |\ X_{n}(x)> y)\frac{1-e^{-y^{-1}}}{1-e^{-y^{-1}-y^{-\alpha(x)}}}P(Z_n(x)\leq y^{\alpha(x)},Z_{n+r}(x')\leq y^{\alpha(x')}),
\end{align*}

\noindent where we have
\begin{align*}\nonumber
\lefteqn{\hspace{-4cm}\lim_{y\to\infty}P(X_{n+r}(x')\leq y\ |\ X_{n}(x)> y)\frac{1-e^{-y^{-1}}}{1-e^{-y^{-1}-y^{-\alpha(x)}}}P(Z_n(x)\leq y^{\alpha(x)},Z_{n+r}(x')\leq y^{\alpha(x')})}\\[0.3cm]
&\hspace{-3cm}= \left\{
\begin{array}{ll}
0&,\,\alpha(x)<1\\
\frac{1}{2}(1-\lambda(X_{n+r}(x')|X_{n}(x)))&,\,\alpha(x)=1\\
1-\lambda(X_{n+r}(x')|X_{n}(x))&,\,\alpha(x)>1
\end{array}.
\right.
\end{align*}	

If $r>0$ then $Z_n(x)$ and $Z_{n+r}(x')$ are independent and
\begin{equation}\nonumber
\lim_{y\to\infty}\frac{1-\frac{P(Z_n(x)\leq y^{\alpha(x)},Z_{n+r}(x')\leq y^{\alpha(x')})}{e^{-y^{-\alpha(x')}}}}{1-e^{-y^{-1}-y^{-\alpha(x)}}}= \displaystyle\lim_{y\to\infty}\frac{1-e^{-y^{-\alpha(x)}}}{1-e^{-y^{-1}-y^{-\alpha(x)}}}=\left\{
\begin{array}{ll}
1&,\,\alpha(x)<1\\
\frac{1}{2}&,\, \alpha(x)=1\\
0&,\, \alpha(x)>1\\
\end{array}.
\right.
\end{equation}	

If $r=0$ and $x\not=x',$ then
\begin{align*}\nonumber
\lefteqn{\hspace{-1cm}\lim_{y\to\infty}e^{-y^{-1}-y^{-\alpha(x')}}\frac{1-\frac{P(Z_n(x)\leq y^{\alpha(x)},Z_{n+r}(x')\leq y^{\alpha(x')})}{e^{-y^{-\alpha(x')}}}}{1-e^{-y^{-1}-y^{-\alpha(x)}}}}\\[0.3cm]
&=\lim_{y\to\infty}e^{-y^{-1}-y^{-\alpha(x')}}\frac{1-\frac{P(Z_{n+r}(x')\leq y^{\alpha(x')})-P(Z_n(x)> y^{\alpha(x)},Z_{n+r}(x')\leq y^{\alpha(x')})}{e^{-y^{-\alpha(x')}}}}{1-e^{-y^{-1}-y^{-\alpha(x)}}}\\[0.3cm]
&=  \frac{P(Z_{n+r}(x')\leq y^{\alpha(x')}|Z_n(x)> y^{\alpha(x)})(1-e^{-y^{-\alpha(x)}})}{1-e^{-y^{-1}-y^{-\alpha(x)}}}e^{-y^{-1}}\\[0.3cm]
&=\left\{
\begin{array}{ll}
1-\lambda(Z_n(x')^{1/\alpha(x')}|Z_n(x)^{1/\alpha(x)})&,\,\alpha(x)<1\\
\frac{1}{2}(1-\lambda(Z_n(x')^{1/\alpha(x')}|Z_n(x)^{1/\alpha(x)}))&,\, \alpha(x)=1\\
0&,\, \alpha(x)>1
\end{array}.
\right.
\end{align*}	
\end{proof}

The previous result shows that the pMAX random field $\{\mathbf{Y}_n\}_{n\geq 1}$ is temporal or spatial-temporal upper tail independent when $\alpha(x)<1,$ whereas upper tail dependence is regulated by the upper tail dependence of $\{\mathbf{X}_n\}_{n\geq 1}.$ On the other hand, spatial upper tail dependence is regulated by the spatial dependence structure of $\{\mathbf{X}_n\}_{n\geq 1}$ and $\{\mathbf{Z}_n\}_{n\geq 1}.$

The case $r>0$ above was already stated in Ferreira and Ferreira (\cite{Ferreira2014} 2014; Proposition 2.5).

\section{Pre-asymptotic dependence}\label{Ferreira_sdepespac}

Previously we have seen that tail dependence can be assessed through coefficient $\lambda(Y_{n+r}(x')|Y_{n}(x))$, where a unit value means total dependence and a null value corresponds to tail independence. In the latter case, there may be a residual or pre-asymptotic dependence captured by the velocity of convergence of the limit in (\ref{Ferreira_lambda}) towards zero. More precisely, if
\begin{eqnarray}\label{Ferreira_etax}
P(X_{n+r}(x')>y\ |\ X_n(x)>y)\sim y^{-1/\eta_{(x'|x)}^{(X,r)}-1}\mathcal{L}_{(x'|x)}^{(X,r)}(y)
\end{eqnarray}
as $y\to\infty$, where function $\mathcal{L}_{(x'|x)}^{(X,r)}(y)$ is slow varying ($\mathcal{L}(y)$ is a slow varying function if $\mathcal{L}:\mathbb{R}^+\to\mathbb{R}$ and $\mathcal{L}(ty)/\mathcal{L}(y)\to 1$, as $y\to\infty$, $t>0$), we say that $\eta_{(x'|x)}^{(X,r)}\in(0,1]$ is a residual tail dependence coefficient usually denoted in literature as Ledford and Tawn tail dependence coefficient (see, e.g., Ledford and Tawn \cite{Tawn1996} 1996, Heffernan \emph{et al.} \cite{Heffernan2007} 2007, Ferreira and Ferreira \cite{Ferreira2014} 2014 and references therein). Notation $a(y)\sim b(y)$ in (\ref{Ferreira_etax}) stands for $a(y)/b(y)\to 1$, as $y\to\infty$. Observe that $\eta_{(x'|x)}^{(X,r)}$ is a space-time measure if $r>0$ and $x\not =x'$, a temporal measure if $r>0$ and $x=x'$ and a spatial measure if $r=0$ and $x\not =x'$ (we denote $\eta_{(x'|x)}^{(X)}\equiv\eta_{(x'|x)}^{(X,0)}$). Observe that tail dependence ($\lambda>0$) corresponds to $\eta=1$ and $\eta <1$ means asymptotic tail independence. Under exact independence we have $\eta=1/2$. A classical example is assigned to Gaussian random pairs with correlation coefficient $\rho$, where $\eta=(1+\rho)/2$.\\

\begin{pro}
In a pMAX random field satisfying (\ref{Ferreira_etax}) with $r>0$, we have
$$
\eta_{(x'|x)}^{(Y,r)}=\left\{
\begin{array}{ll}
\alpha(x)\max\left(
\eta_{(x'|x)}^{(X,r)},\frac{1}{1+\min(\alpha(x'),1)}\right) & ,\,\alpha(x)<1\\\\
\max\left(\eta_{(x'|x)}^{(X)},\frac{1}{1+\alpha(x)},\frac{1}{1+\alpha(x')}\right) & ,\,\alpha(x)\geq 1
\end{array}.
\right.
$$
If $r=0$ and, in addition,
\begin{eqnarray}\label{Ferreira_etaz}
P(Z_{n}(x')^{1/\alpha(x')}>y\ |\ Z_n(x)^{1/\alpha(x)}>y)\sim y^{-1/\eta_{(\alpha(x')|\alpha(x))}^{(Z)}-1}\mathcal{L}_{(\alpha(x'),\alpha(x))}^{(Z)}(y)
\end{eqnarray}
holds, as $y\to\infty$, where $\mathcal{L}_{(\alpha(x'),\alpha(x))}^{(Z)}(y)$ is a slowly varying function and $\eta_{(\alpha(x')|\alpha(x))}^{(Z)}\in(0,1]$, then
$$
\eta_{(x'|x)}^{(Y)}=
\min\left(1,\alpha(x)\right)\max\left(
\eta_{(x'|x)}^{(X)},\eta_{(\alpha(x')|\alpha(x))}^{(Z)},\frac{1}{1+\alpha(x)},
\frac{1}{1+\alpha(x')}\right).
$$
\end{pro}
\begin{proof}
Observe that
\begin{eqnarray*}\nonumber
\lefteqn{P(Y_{n+r}(x')>y,Y_n(x)>y)}\\[0.3cm]
&=&P(X_{n+r}(x')>y,X_n(x)>y)\left(1-P(Z_{n+r}(x')>y^{\alpha(x')})-P(Z_n(x)>y^{\alpha(x)})
\right)\\[0.3cm]
&&+\ P(Z_{n+r}(x')>y^{\alpha(x')},Z_n(x)>y^{\alpha(x)})\left(1-P(X_{n+r}(x')>y)-P(X_n(x)>y)\right)\\[0.3cm]
&&+\ P(X_{n+r}(x')>y,X_n(x)>y)P(Z_{n+r}(x')>y^{\alpha(x')},Z_n(x)>y^{\alpha(x)})\\[0.3cm]
&&+\ P(X_n(x)>y)P(Z_{n+r}(x')>y^{\alpha(x')})+P(X_{n+r}(x')>y)P(Z_n(x)>y^{\alpha(x)})
\end{eqnarray*}
and thus, as $y\to\infty$,
\begin{eqnarray*}
\lefteqn{P(Y_{n+r}(x')>y,Y_n(x)>y)}\\[0.3cm]
&\sim&P(X_{n+r}(x')>y,X_n(x)>y)+P(Z_{n+r}(x')>y^{\alpha(x')},Z_n(x)>y^{\alpha(x)})
+y^{-1}\left(y^{-\alpha(x')}+y^{-\alpha(x)}\right).
\end{eqnarray*}
If $r>0$, under condition (\ref{Ferreira_etax}), we have
$$P(Y_{n+r}(x')>y,Y_n(x)>y)
\sim y^{-1/\eta_{(x'|x)}^{(X)}}\mathcal{L}_{(x'|x)}^{(X)}(y)+y^{-\alpha(x')-\alpha(x)}
+y^{-1}\left(y^{-\alpha(x')}+y^{-\alpha(x)}\right).
$$
and therefore
$$
\frac{P(Y_{n+r}(x')>y,Y_n(x)>y)}{P(Y_n(x)>y)^{1/\eta_{(x'|x)}^{(Y)}}}
\sim\displaystyle\frac{y^{-1/\eta_{(x'|x)}^{(X)}}\mathcal{L}_{(x'|x)}^{(X)}(y)+
y^{-\alpha(x')-\alpha(x)} +y^{-1-\alpha(x')}+y^{-1-\alpha(x)}}{\left(y^{-1}+y^{-\alpha(x)}\right)^{1/\eta_{(x'|x)}^{(Y)}}}.
$$
If $r=0$, just observe that, by (\ref{Ferreira_etax}) and (\ref{Ferreira_etaz}), we have
\begin{eqnarray*}
\lefteqn{P(Y_n(x')>y,Y_n(x)>y)}\\[0.3cm]
&\sim& y^{-1/\eta_{(x'|x)}^{(X)}}\mathcal{L}_{(x'|x)}^{(X)}(y)+y^{-1/\eta_{(\alpha(x')
|\alpha(x))}^{(Z)}}\mathcal{L}_{(\alpha(x')|\alpha(x))}^{(Z)}(y)+y^{-1}\left(y^{-\alpha(x')}
+y^{-\alpha(x)}\right).
\end{eqnarray*}
\end{proof}

The case $r>0$ in the previous result was already derived in Ferreira and Ferreira (\cite{Ferreira2014} 2014; Proposition 2.6)

\section{Examples}\label{Ferreira_sexemp}

For the pMAX model defined in (\ref{Ferreira_pmax}) we shall now consider particular cases of $\{X_n(x), x\in \mathbb{R}^2\}_{n\geq 1}$ and $\{Z_n(x), x\in \mathbb{R}^2\}_{n\geq 1}$ to illustrate the several types of dependence captured by the previously presented coefficients. $\{X_n(x), x\in \mathbb{R}^2\}_{n\geq 1}$ will be taken as  a sequence of moving maxima random fields, first of a sequence of i.i.d. random fields with unit Fréchet margins and second  of a sequence of i.i.d. random fields with unit Fréchet margins exhibiting spatial dependence. In the latter case the dependence between the margins are given by the  Schlather model (Schlather \cite{sha2002} 2002).

\begin{ex}
Let $\{\widehat{X}_n(x),\,x\in\mathbb{R}^2\}_{n\geq 1}$ be an i.i.d. sequence of random fields of unit Fréchet independent margins. Consider
\begin{equation*}
X_n(x)=\frac{2}{3}\hat{X}_n(x)\vee \frac{1}{3}\hat{X}_{n-1}(x),\,n\geq 1,
\end{equation*}
and  $\{Z_n(x)=Z_n,\,x\in\mathbb{R}^2\}_{n\geq 1}$. Thus we have a stationary sequence of random fields $\{X_n(x),\,x\in\mathbb{R}^2\}_{n\geq 1}$ with independent margins and $\{Z_n\}_{n\geq 1}$ i.i.d.. The marginal distributions $F_{X_n}(x)$ remain standard Fr\'echet.

In each location $x\in \mathbb{R}^2,$ the temporal dependence of $\{Y_n(x)\}_{n\geq 1}$ will be induced by the $1$-dependence of $\{X_n(x)\}_{n\geq 1},$ whereas the spatial dependence will be induced by common $Z_n$ to every locations, regulated by $\alpha(x)$. More precisely, we have temporal dependence given by
\begin{equation}\label{lamb}
\lambda(Y_{n+r}(x)|Y_n(x))=\left\{
\begin{array}{ll}
0 & ,\, \alpha(x)<1\vee r>1\\
1/6 & ,\, \alpha(x)=1\wedge r=1\\
1/3 & ,\, \alpha(x)> 1\wedge r=1,
\end{array}
\right.
\end{equation}
space-time dependence
$$
\lambda(Y_{n+r}(x')|Y_n(x))=0,\, r\geq 1,\, n\geq 1,
$$
for $x\not=x'$, and spatial dependence
$$
\lambda(Y_{n}(x')|Y_n(x)) =\left\{
\begin{array}{ll}
1 & ,\, \alpha(x')\leq \alpha(x)<1\\
1/2 & ,\, \alpha(x')\leq \alpha(x)=1\\
0 & ,\, \textrm{otherwise.}
\end{array}
\right.
$$\vspace{0.3cm}

Since the tail dependence coefficient is invariant to increasing transformations one can rather consider $\lambda(F(Y_{n+r}(x'))|F(Y_n(x))),$ $n\geq 1,$ $r\geq 0,$ $x,x'\in \mathbb{R}^2,$  where $F$ is the standard Fréchet distribution function. Observations of pairs $(F(Y_n(x)),F(Y_{n+r}(x'))),$ $n\geq 1,$ $r\geq 0,$ $x, x'\in \mathbb{R}^2,$ are plotted in Figures \ref{fig1}, \ref{fig2} \ref{fig3} to highlight, respectively, the spatial ($r=0$ and $x\neq x'$), temporal ($r>0$ and $x=x'$) and spatial-temporal ($r>0$ and $x\neq x'$) dependence.

\begin{figure}[!ht]
 \begin{center}
 \includegraphics[scale=0.35]{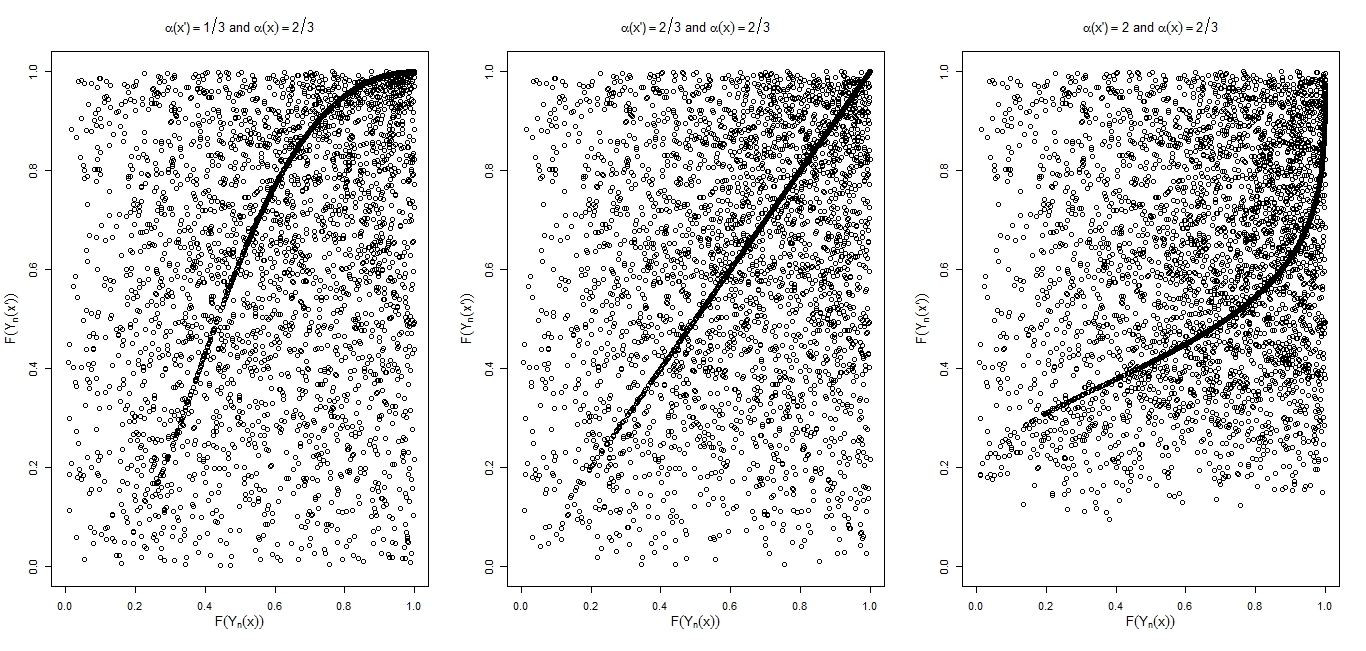}
 \includegraphics[scale=0.35]{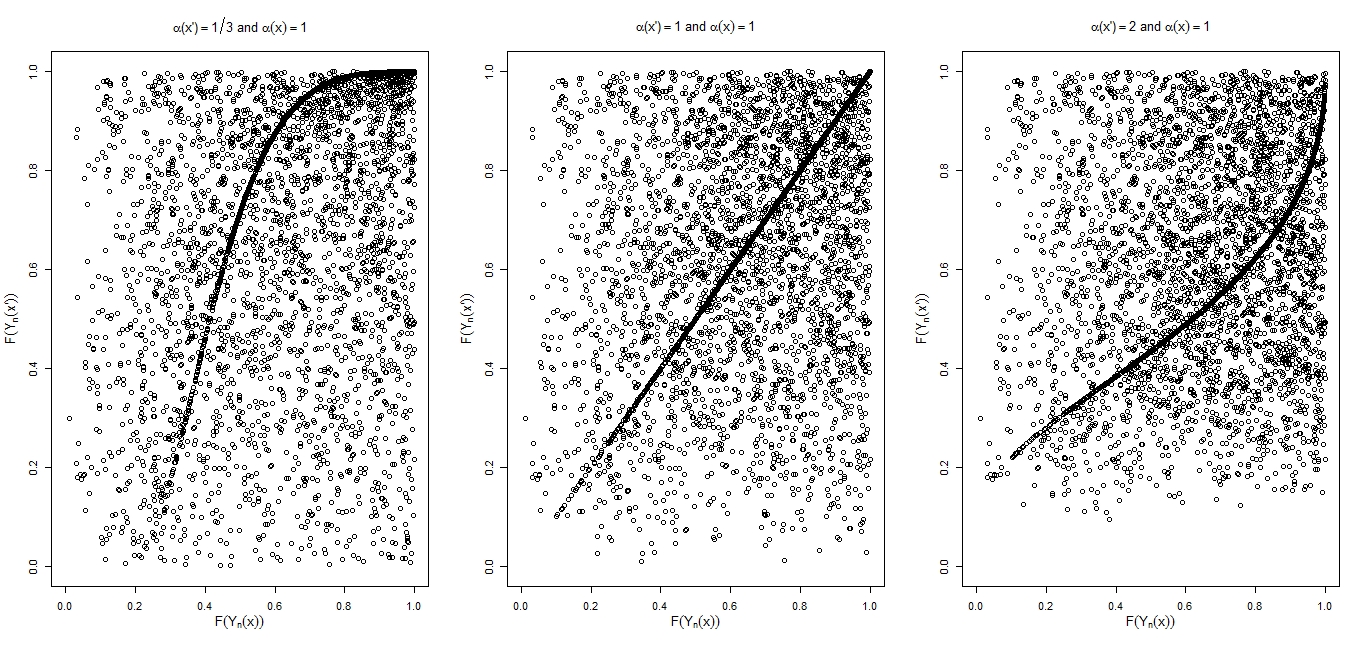}
 \includegraphics[scale=0.35]{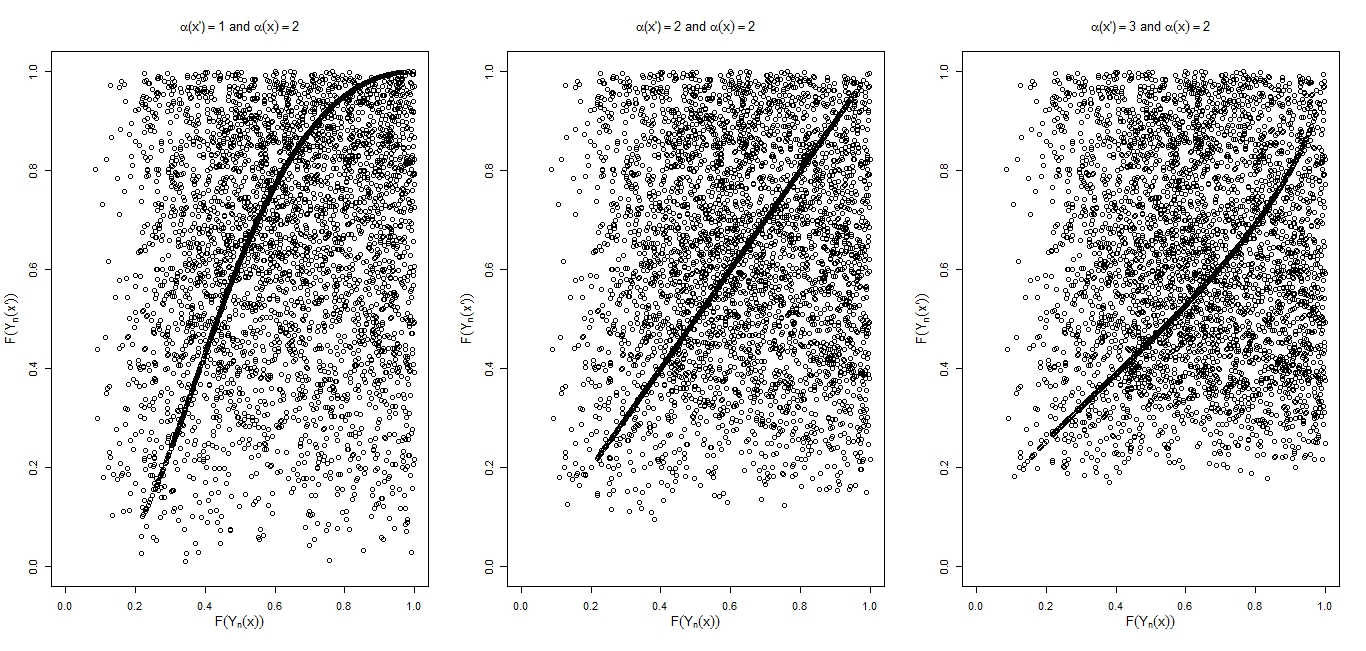}
 \caption{Observations of $(F(Y_n(x)), F(Y_n(x')))$ $n\geq 1,$  $x, x'\in \mathbb{R}^2,$ $x\neq x'$ for different values of $\alpha(x')$ and $\alpha(x)$}
  \end{center}\label{fig1}
 \end{figure}

 \begin{figure}[!ht]
 \begin{center}
 \includegraphics[scale=0.35]{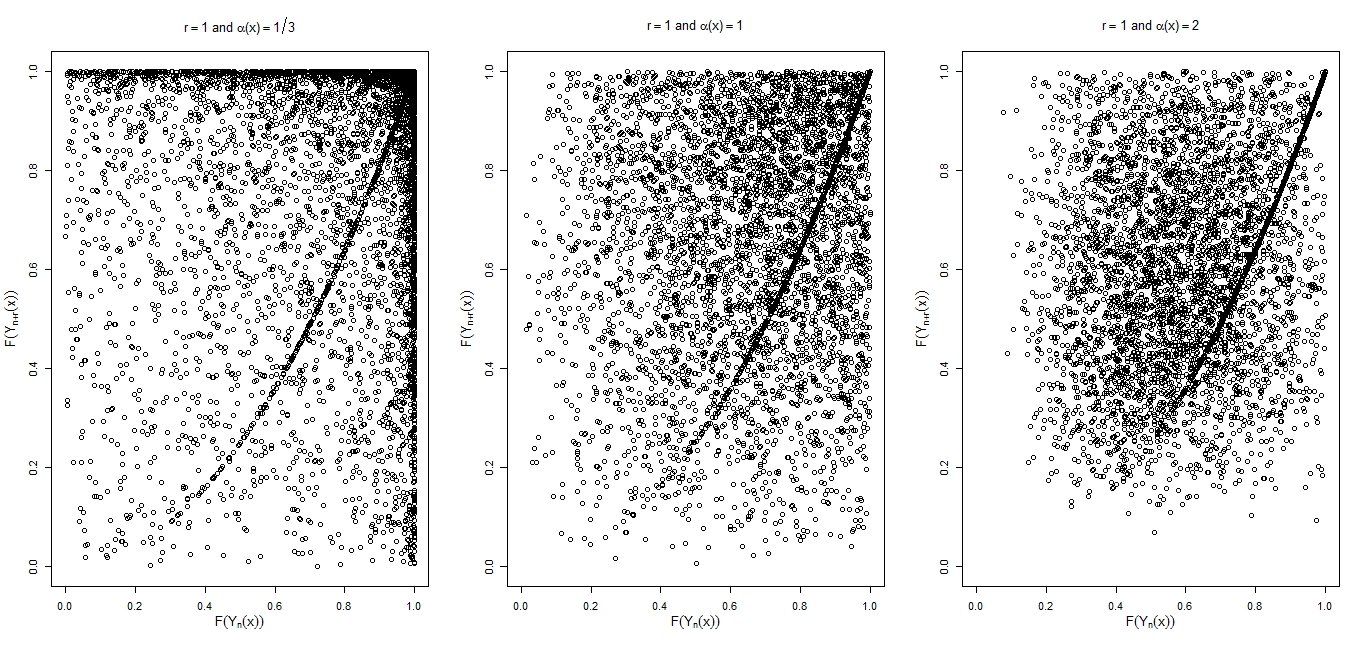}
 \includegraphics[scale=0.35]{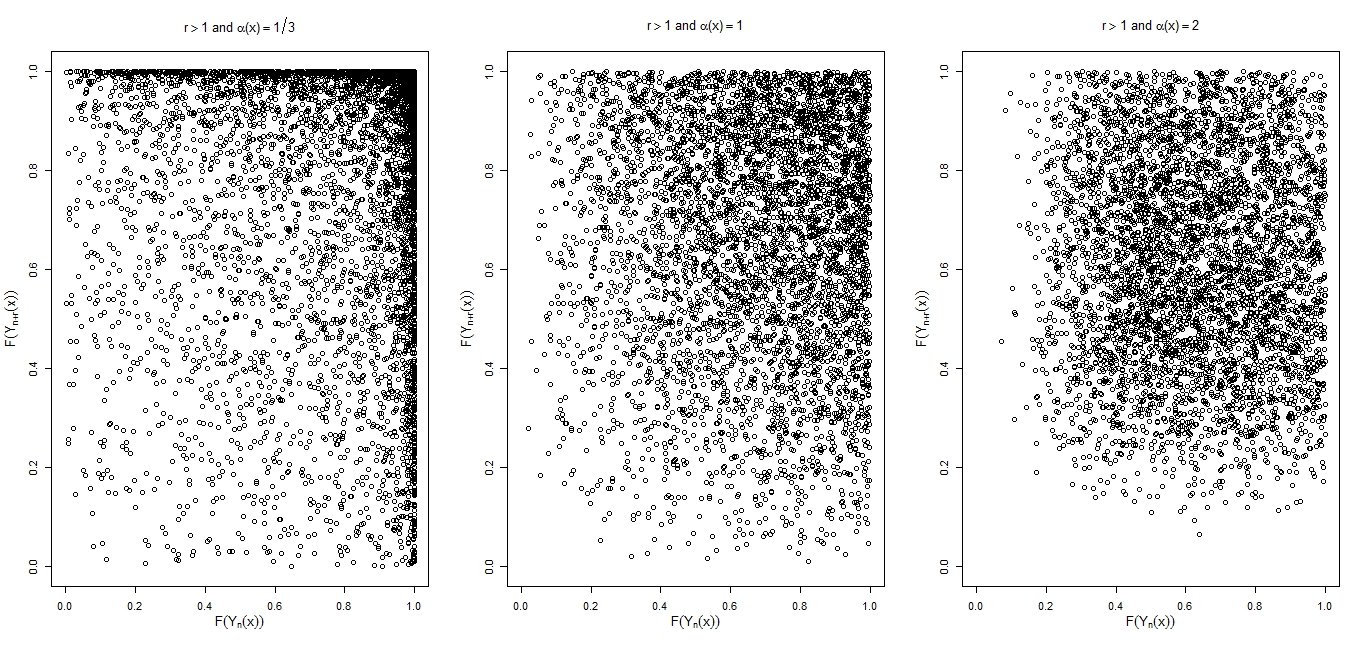}
 \caption{Observations of $(F(Y_n(x)), F(Y_{n+r}(x)))$ $n\geq 1,$ $r\geq 1,$  $x\in \mathbb{R}^2,$ for different values  $\alpha(x)$ }
  \end{center}\label{fig2}
 \end{figure}

 \begin{figure}[!ht]
 \begin{center}
 \includegraphics[scale=0.35]{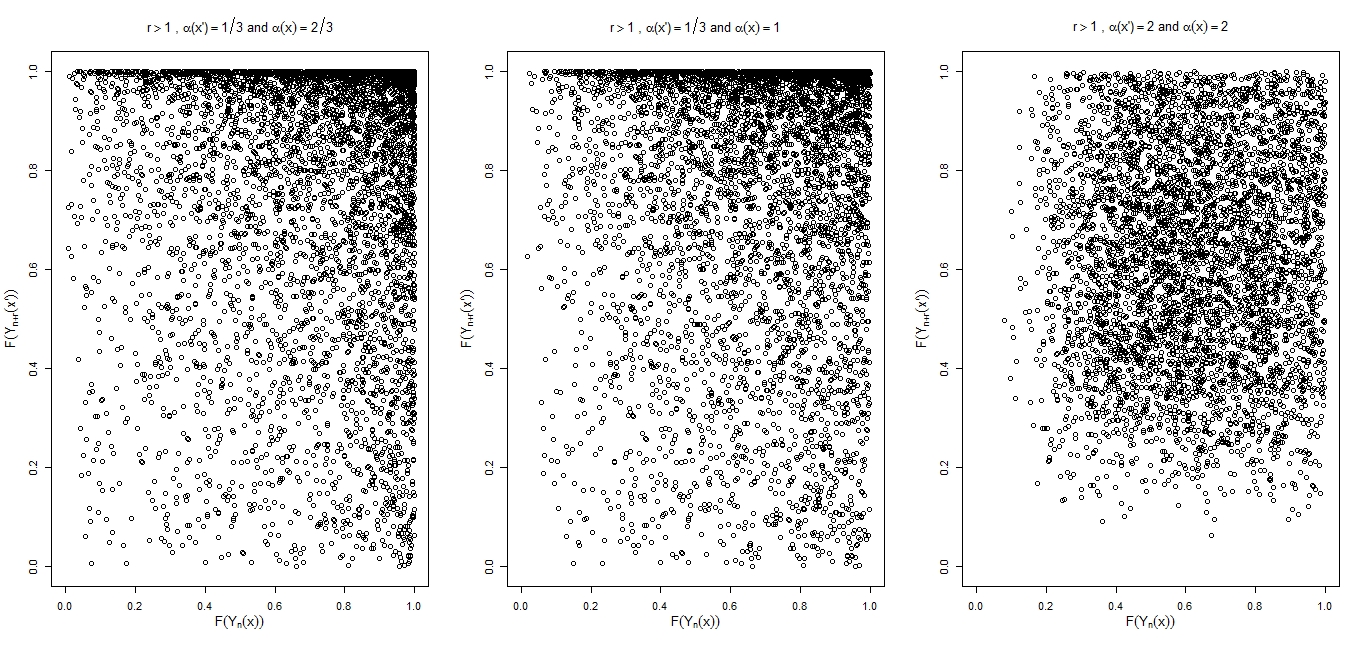}
 \caption{Observations of $(F(Y_n(x)), F(Y_{n+r}(x')))$ $n\geq 1,$ $r> 1,$  $x, x'\in \mathbb{R}^2,$ for different values of  $\alpha(x')$ and $\alpha(x)$}
  \end{center}\label{fig3}
 \end{figure}

For pre-asymptotic dependence, we have temporal
$$
\eta_{(x)}^{(Y,r)}=\left\{
\begin{array}{ll}
\max(1/2,\alpha(x)) & ,\, \alpha(x)<1\wedge r=1\\
1 & ,\,  \alpha(x)\geq 1\wedge r=1\\
1/2 & ,\, \textrm{otherwise.}\\
\end{array}
\right.
$$
space-time $\eta_{(x'|x)}^{(Y,r)}=1/2$ and spatial
$$
\eta_{(x'|x)}^{(Y)}=\left\{
\begin{array}{ll}
1 & ,\,\alpha(x')\leq\alpha(x)\leq 1\\
\frac{\alpha(x)}{1+\alpha(x)} & ,\,\alpha(x)<1<\alpha(x')\\
\frac{\alpha(x)}{\alpha(x')} & ,\,\alpha(x)<1,\alpha(x)<\alpha(x')<1+\alpha(x)\\
\max\left(\frac{1}{2},\frac{1}{\alpha(x')}\right) & ,\,1<\alpha(x)<\alpha(x')\\
\frac{1}{1+\alpha(x')} & ,\, \alpha(x')<1<\alpha(x)\\
\frac{1}{\alpha(x)} & ,\,\alpha(x')<1<\alpha(x)<1+\alpha(x')\\
\max\left(\frac{1}{2},\frac{1}{\alpha(x)}\right) & ,\, 1<\alpha(x')<\alpha(x)\,.
\end{array}
\right.
$$
\end{ex}

\begin{ex}
Let us now modify the previous example, so that $\{\widehat{X}_n(x),\,x\in\mathbb{R}^2\}_{n\geq 1}$  is an i.i.d. sequence of  random fields with unit Fréchet margins exhibiting  spatial dependence. More precisely, with the dependence between the margins $\widehat{X}_n(x)$ and $\widehat{X}_n(x')$ given by the  Schlather model. Hence, according to Schlather (\cite{sha2002} 2002), $\widehat{X}_n(x)\stackrel{d}{=} \widehat{X}(x)$ where  $$\widehat{X}(x)=\max_{i\geq 1}\xi_i\max\{0,U_i(x)\},\quad x\in \mathbb{R},$$ with $\{U_i(x),x \in \mathbb{R}\}_{i\geq 1}$ a sequence of independent stationary  standard Gaussian processes. For each $i\geq 1,$  $\{U_i(x),x \in \mathbb{R}\}_{i\geq 1}$ has  correlation function $\rho(h)$, scaled
so that $E[\max\{0, Y_i(x)\}] = 1$ and $\{\xi_i\}_{i\geq 1}$ are the points of a Poisson process on $\mathbb{R}^+$ with intensity measure $\xi^{-2}d\xi.$ It is a  max-stable process (de Haan \cite{haan1984} 1984).

Once again we shall consider $\{X_n(x)=\frac{2}{3}\hat{X}_n(x)\vee \frac{1}{3}\hat{X}_{n-1}(x), x\in \mathbb{R}^2\}_{n\geq 1}$
and  $\{Z_n(x)=Z_n,\,x\in\mathbb{R}^2\}_{n\geq 1}$ independent of the previous sequence. For each location $x\in \mathbb{R}^2,$  the temporal dependence of $\{Y_n(x)=X_n(x)\vee Z_n^{1/\alpha(x)}\}_{n\geq 1}$ is regulated by the 1-dependence of $\{X_n(x)\}_{n\geq 1},$ while the spatial dependence is induced by the common value of $Z_n$ in every location and by the correlation $\rho(h),$ where $h\in \mathbb{R}^+$ is the Euclidean distance between locations $x$ and $x'.$

For the temporal dependence  we find again $\lambda(Y_{n+r}(x)|Y_n(x)),$ $r\geq 1,$ given in (\ref{lamb}), as expected. As to what concerns space-time dependence, for $x\neq x',$ $r>1,$ $n\geq 1,$ $\lambda(Y_{n+r}(x')|Y_n(x))=0,$ from the 1-dependence of $\{{\bf{Y}}_n\}_{n\geq 1}.$
On the other hand, since $$P(\widehat{X}_n(x)\leq z, \widehat{X}_n(x')\leq z')=\exp\left[-\frac{1}{2}\left(\frac{1}{z}+\frac{1}{z'}\right)
\left(1+\sqrt{1-2(\rho(h)+1)\frac{zz'}{(z+z')^2}}\right)\right],$$  for $x, x'\in \mathbb{R}^2$ and $h=\|x-x'\|,$ we find, for $x\neq x',$ $n\geq 1,$
\begin{equation}\label{lamb2}
\lambda(Y_{n+1}(x')|Y_n(x))=\left\{
\begin{array}{ll}
0 & ,\,\alpha(x)< 1\\
\frac{1}{2}\left(\frac{1-\sqrt{1-\frac{4}{9}(\rho(h)+1)}}{2}\right) & ,\,\alpha(x)=1\\
\frac{1-\sqrt{1-\frac{4}{9}(\rho(h)+1)}}{2}& ,\,\alpha(x)>1
\end{array}
\right.
\end{equation}
and spatial dependence given by
\begin{eqnarray*}
\lambda(Y_{n}(x')|Y_n(x))&=&\lim_{y\to \infty}\left[\frac{1}{1+y^{1-\alpha(x)}}\left(\frac{1-\sqrt{1-\frac{1}{2}(\rho(h)+1)}}{2}\right)
+\frac{y^{-(\alpha(x')\vee \alpha(x))}}{y^{-1}+y^{-\alpha(x)}}\right]\\
&=&\left\{
\begin{array}{ll}
\frac{1-\sqrt{1-\frac{1}{2}(\rho(h)+1)}}{4}+\frac{1}{2} & ,\,\alpha(x)= 1\ \wedge \ \alpha(x')<1 \\
\frac{1-\sqrt{1-\frac{1}{2}(\rho(h)+1)}}{4} & ,\,\alpha(x)= 1\ \wedge \ \alpha(x')>1 \\
0 & ,\,\alpha(x)<1\ \wedge\ \alpha(x')>\alpha(x)\\
1& ,\,\alpha(x)<1\ \wedge\ \alpha(x')\leq\alpha(x)\\
\frac{1-\sqrt{1-\frac{1}{2}(\rho(h)+1)}}{2}& ,\,\alpha(x)>1
\end{array}.
\right.
\end{eqnarray*}

When $x=x'$ the result in  (\ref{lamb2}) leads again to (\ref{lamb}).


In the forthcoming simulations we  shall consider $\rho(h)$ belonging to the powered exponential parametric family, i.e. $\rho(h)=\exp[-(h/c_2)^{\nu}],$ where  $c_2\in ]0,+\infty[$ is the range parameter and $\nu \in ]0,2]$ is the smooth parameter, more precisely we shall consider $\rho(h)=\exp(-h).$

\end{ex}

\section{Estimation of model parameters}

The pMAX random field defined in  (\ref{Ferreira_pmax}) sets on parameters $\alpha(x)>0$ which vary with locations $x\in \mathbb{R}^2.$ The estimation of these parameters is essential for practical applications of the model. We shall therefore present a way to estimate $\alpha(x)$ at a given location $x\in \mathbb{R}^2.$\vspace{0.3cm}

As previously pointed out, for any location $x\in \mathbb{R}^2$ the distribution function of $Y_n(x),$ $n\geq 1,$ is given by $F(z)\equiv F_{Y_n(x)}(z)=e^{-z^{-1}-z^{-\alpha(x)}},$ for all $z>0.$ We can then write
\begin{equation}
\alpha(x)=\frac{\ln(-\ln(F(z))-\frac{1}{z})}{\ln\left(\frac{1}{z}\right)}, \ x\in\mathbb{R}^2,\ 0<z\neq 1.\label{alfa}
\end{equation}

We point out that parameter $\alpha(x)$ ``controls'' the tail of the distribution $F(z),$ $z>0,$ of  $Y_n(x),$ $n\geq 1.$ Thus, values of $\alpha(x)$ smaller than one lead to lighter tail distributions than values greater or equal to one.

Expression (\ref{alfa})  provides a simple way to estimate the parameter  $\alpha(x),$ $x\in \mathbb{R}^2,$ involved in model (\ref{Ferreira_pmax}). Therefore, if $(Y_1,\ldots,Y_n)$ is a random sample of $Y_n(x),$ $n\geq 1,$ for a given $x\in \mathbb{R}^2,$ $\alpha(x)>0$ can be estimated with
\begin{equation}\label{estim_alfa}
\widehat{\alpha}_n(x)=\frac{1}{n}\sum_{i=1}^n\frac{\ln(-\ln(\widehat{F}_{n}(Z_i))-1/Z_i)}{\ln(1/Z_i)}
\end{equation}
where  $\widehat{F}_n(y),$ $z\in \mathbb{R},$ denotes the empirical distribution function, associated to the $n-$sample $(Y_1,\ldots,Y_n),$ defined as $$\widehat{F}_n(y)=\frac{1}{n}\sum_{j=1}^n \indi_{\{Y_j\leq y\}},$$ with $\indi_A$  the indicator of an event $A.$ The $n$ values $(Z_1,\ldots,Z_n)$ are  such that $Z_i>1$ and $\widehat{F}_n(Z_i)< 1,$ $i=1,\ldots,n.$

In what follows, we explore through simulations the finite sample properties of the proposed estimator (\ref{estim_alfa}) for parameter $\alpha(x)$ at a given location $x\in \mathbb{R}^2.$

Each simulated data set consists of 1000 independent copies of $n$ realizations of a random sample $(Y_1,\ldots,Y_n)$ of (\ref{Ferreira_pmax}), with $X_n(x)$ and $Z_n(x),$ at a location $x\in \mathbb{R}^2,$ defined as in Examples 5.1 and 5.2, for different values of $\alpha(x)>0.$ Four different sample sizes are considered for each data set. The sample means $\widehat{\mu}(\widehat{\alpha}_{n,i}(x))$ and the sample standard deviations $\widehat{\sigma}(\widehat{\alpha}_{n,i}(x))$ of the estimates $\widehat{\alpha}_{n,i}(x),$ $i=1,\ldots,1000,$ depending on the sample size $n,$ were computed. The bias and the root mean squared errors ($RMSE(\widehat{\alpha}_{n,i}(x))$) were also determined.  The $n$ values $(Z_1,\ldots,Z_n)$ were considered to be an arithmetic sequence of $n$ numbers from 1.1 to the $kth$ percentile of $(Y_1,\ldots,Y_n)$, with $k=95$ and $k=75$.  Tables \ref{tab1} and \ref{tab2} summarize the estimation results obtained for the two percentile choices, respectively the $95th$ and the $75th$ percentile.

\begin{table}[t!]
\caption{Estimation results for $\alpha(x)>0$ in model (\ref{Ferreira_pmax}) of Example 5.1}
\begin{center}
\begin{tabular}{lL{2cm}ccC{2cm}c}
\hline \multicolumn{6}{c}{$95th$ percentile}\\
\hline
 $\alpha(x)$ & n& $\widehat{\mu}(\widehat{\alpha}_{n}(x))$ & $BIAS(\widehat{\alpha}_{n}(x))$ & $\widehat{\sigma}(\widehat{\alpha}_{n}(x))$&$RMSE(\widehat{\alpha}_{n}(x))$\\ \hline
 0.10& $n=100$&0.1010&0.0010& 0.0306&0.0306\\
 &$n=500$& 0.0995& -0.0005& 0.0071& 0.0071\\
 & $n=1000$& 0.1002& 0.0002& 0.0048& 0.0048\\
 & $n=5000$& 0.1000& $3.59\times 10^{-5}$&0.0020& 0.0020\\
 0.50& $n=100$& 0.5027& 0.0027& 0.0808&0.0808\\
 & $n=500$& 0.4981& -0.0019& 0.0317& 0.0317\\
 & $n=1000$& 0.4994& -0.0006& 0.0221& 0.0221\\
 & $n=5000$& 0.5000& $2.34\times 10^{-5}$& 0.0095& 0.0095\\
 1.00& $n=100$& 1.0212& 0.0212& 0.2108& 0.2118\\
 & $n=500$&  1.0079& 0.0079& 0.0884& 0.0887\\
 & $n=1000$& 1.0041 & 0.0041& 0.0604& 0.0605\\
 & $n=5000$&  1.0003& 0.0003& 0.0254&0.0254\\
 1.50& $n=100$&1.5308 & 0.0308& 0.4014& 0.4024 \\
 & $n=500$& 1.5387& 0.0387& 0.2243& 0.2275\\
 & $n=1000$&  1.5253& 0.0253& 0.1658& 0.1676\\
 & $n=5000$&  1.5045& 0.0045 &0.0711& 0.0712\\
 2.00& $n=100$& 1.9379& -0.0621& 0.6186& 0.6214\\
 & $n=500$&  1.9982 &  -0.0018& 0.3579& 0.3577\\
 & $n=1000$&   2.0063 & 0.0063& 0.2971& 0.2970\\
 & $n=5000$&  2.0327& 0.0327& 0.1916& 0.1943\\
  \hline \hline
  \multicolumn{6}{c}{$75th$ percentile}\\
\hline
$\alpha(x)$ & n& $\widehat{\mu}(\widehat{\alpha}_{n}(x))$ & $BIAS(\widehat{\alpha}_{n}(x))$ & $\widehat{\sigma}(\widehat{\alpha}_{n}(x))$&$RMSE(\widehat{\alpha}_{n}(x))$\\ \hline
 0.10& $n=100$& 0.1008&0.00081&  0.0367& 0.0367\\
 &$n=500$& 0.1000& $4.42\times 10^{-5}$&0.0084&0.0084\\
 & $n=1000$& 0.1000& $4.85\times 10^{-5}$&0.0056& 0.0056\\
 & $n=5000$& 0.1000& $2.39\times 10^{-5}$ & 0.0024& 0.0024\\
 0.50& $n=100$& 0.5060 &0.0060& 0.1319& 0.1320\\
 & $n=500$&  0.5040&0.0040& 0.0527& 0.0528\\
 & $n=1000$& 0.5012& 0.0012& 0.0372& 0.0372\\
 & $n=5000$& 0.5008& 0.0008& 0.0161& 0.0161\\
 1.00& $n=100$& 1.0387& 0.0387& 0.2927& 0.2951\\
 & $n=500$&  1.0051& 0.0051& 0.1173& 0.1174\\
 & $n=1000$& 1.0003& 0.0003& 0.0806& 0.0805\\
 & $n=5000$&  1.0001 & $9.61\times 10^{-5}$&0.0349& 0.0349\\
 1.50& $n=100$&1.5452& 0.0452& 0.4753& 0.4772 \\
 & $n=500$& 1.5124& 0.0124& 0.1866& 0.1869\\
 & $n=1000$&  1.5029& 0.0029& 0.1289&0.1289\\
 & $n=5000$&  1.4994& -0.0006& 0.0566& 0.0566\\
 2.00& $n=100$& 2.1143&0.1143& 0.6962& 0.7052\\
 & $n=500$&  2.0322& 0.0322& 0.2734& 0.2752\\
 & $n=1000$&   2.0222& 0.0222& 0.1914& 0.1926\\
 & $n=5000$&  2.0005& 0.0005& 0.0843& 0.0843\\
  \hline
 \end{tabular}\label{tab1}
\end{center}
\end{table}

\begin{table}[t!]
\caption{Estimation results for $\alpha(x)>0$ in model (\ref{Ferreira_pmax}) of Example 5.2}
\begin{center}
\begin{tabular}{lL{2cm}ccC{2cm}c}
\hline \multicolumn{6}{c}{$95th$ percentile}\\
\hline
 $\alpha(x)$ & n& $\widehat{\mu}(\widehat{\alpha}_{n}(x))$ & $BIAS(\widehat{\alpha}_{n}(x))$ & $\widehat{\sigma}(\widehat{\alpha}_{n}(x))$&$RMSE(\widehat{\alpha}_{n}(x))$\\ \hline
 0.10& $n=100$&0.0994&-0.0006&0.0311&0.0311\\
 &$n=500$&0.1001&0.0001&0.0069&0.0069\\
 & $n=1000$& 0.1001&0.0001&0.0047&0.0047\\
 & $n=5000$&  0.0999&-6.28$\times 10^{-5}$&0.0021&0.0021\\
 0.50& $n=100$&0.5027&0.0027&0.0822&0.0822\\
 & $n=500$&  0.4993&-0.0007&0.03199&0.03120\\
 & $n=1000$& 0.4999&-8.34&0.0208&0.0208\\
 & $n=5000$& 0.5000&2.39$\times 10^{-5}$&0.0096&0.0096\\
 1.00& $n=100$&1.0345&0.0345&0.2294&0.2319\\
 & $n=500$& 1.0057&0.0057&0.0851&0.0853\\
 & $n=1000$& 1.0013&0.0013&0.0576&0.0576\\
 & $n=5000$&  1.0005&0.0005&0.0264&0.0264\\
 1.50& $n=100$&1.5286&0.0286&0.4206&0.4214\\
 & $n=500$& 1.5336&0.0336&0.2315&0.2338\\
 & $n=1000$&  1.5316&0.0316&0.1739&0.1766\\
 & $n=5000$&  1.5048&0.0048&0.0715&0.0716\\
 2.00& $n=100$&1.9551&-0.0449&0.6755&0.6767\\
 & $n=500$&  2.0069&0.0069&0.3582&0.3580\\
 & $n=1000$&   2.0242&0.0242&0.2998&0.3006\\
 & $n=5000$& 2.0378&0.0378&0.1904&0.1941\\
  \hline \hline
  \multicolumn{6}{c}{$75th$ percentile}\\
\hline
$\alpha(x)$ & n& $\widehat{\mu}(\widehat{\alpha}_{n}(x))$ & $BIAS(\widehat{\alpha}_{n}(x))$ & $\widehat{\sigma}(\widehat{\alpha}_{n}(x))$&$RMSE(\widehat{\alpha}_{n}(x))$\\ \hline
 0.10& $n=100$& 0.1000&4.33$\times 10^{-5}$&0.0369&0.0368\\
 &$n=500$&  0.1002&0.0002&0.0086&0.0086\\
 & $n=1000$& 0.1003&0.0003&0.0056&0.0057\\
 & $n=5000$& 0.1001&8.43$\times 10^{-5}$&0.0022&0.0022\\
 0.50& $n=100$& 0.5078&0.0078&0.1366&0.1367\\
 & $n=500$& 0.5013&0.0013&0.0530&0.0530\\
 & $n=1000$& 0.5011&0.0011&0.0365&0.0365\\
 & $n=5000$& 0.5006&0.0006&0.0160&0.0160\\
 1.00& $n=100$& 1.0403&0.0403&0.2974&0.3000\\
 & $n=500$& 1.0011&0.0011&0.1159&0.1159\\
 & $n=1000$&1.0032&0.0032&0.0820&0.0820\\
 & $n=5000$&1.0014&0.0014&0.0371&0.0371\\
 1.50& $n=100$& 1.5511&0.0511&0.4414&0.4442\\
 & $n=500$& 1.5099&0.0099&0.1764&0.1766\\
 & $n=1000$& 1.5037&0.0037&0.1270&0.1269\\
 & $n=5000$&  1.5029&0.0029&0.0571&0.0572\\
 2.00& $n=100$& 2.0734&0.0734&0.6760&0.6797\\
 & $n=500$&  2.0269&0.0269&0.2774&0.2785\\
 & $n=1000$& 2.0119&0.0119&0.1861&0.1864\\
 & $n=5000$&  2.0068&0.0068&0.0843&0.0845\\
  \hline
 \end{tabular}\label{tab2}
\end{center}
\end{table}

The results of Tables \ref{tab1} and \ref{tab2} show that the proposed estimator for $\alpha(x)$ has an overall good performance. Due to the tail weight of $F(z)$ the estimator has a better behavior for $\alpha(x)<1$ (light tail) when the empirical distribution function is evaluated to higher percentiles of the data, whereas for $\alpha(x)\geq 1$ (heavy tail) better results are obtained when the empirical distribution function is evaluated to lower percentiles of the data.

\section{Discussion}

The pMAX model presented in this work is another contribution to the modeling of heavy tail random fields. The $\alpha(x)$ parameter allows for an encompassing extreme dependency structure, including asymptotic and pre-asymptotic dependence. If $\alpha(x)$ is less than or equal to 1, its value affects the measures of dependence, $\lambda(Y_n(x')|Y_n(x))$ and $\eta_{(x'|x)}^{(Y)},$  that were considered here. In this case, $\alpha(x)$ also corresponds to the tail index of the marginal $Y_n(x)$, so it can be estimated as such.


\end{document}